\documentclass[11pt]{amsart}

\RequirePackage[l2tabu,orthodox]{nag}
\usepackage[T1]{fontenc}
\usepackage[utf8]{inputenc}
\usepackage{lmodern}
\usepackage[english]{babel}
\usepackage{color} 
\usepackage[usenames,dvipsnames]{xcolor}
\usepackage{amsrefs}
\usepackage{amsthm,amssymb,amsmath}
\usepackage{url}

\newtheorem{thm}{Theorem}[]
\theoremstyle{remark}
\newtheorem{rem}[thm]{Remark}

\usepackage[margin=2.4cm]{geometry}

\usepackage{indentfirst}
\usepackage{microtype}

\usepackage[colorlinks=true,linkcolor=blue,citecolor=magenta]{hyperref}

\DeclareMathOperator{\homeo}{\mathsf{Homeo}}
\DeclareMathOperator{\stab}{\mathsf{Stab}}
\DeclareMathOperator{\orb}{\mathsf{Orb}}
\DeclareMathOperator{\PSL}{\mathsf{PSL}}

\DeclareMathOperator{\LO}{\mathsf{LO}}

\newcommand{\id}{\mathsf{id}}
\newcommand{\R}{\mathbb{R}}

\newcommand{\Z}{\mathbb Z}

\renewcommand{\setminus}{\smallsetminus}
\renewcommand{\emptyset}{\varnothing}

\title[The $\Sigma(2,3,7)$ homology sphere group and its space of left-orders]{On the action of the $\Sigma(2,3,7)$ homology sphere group on its space of left-orders}

\author{Kathryn Mann \and Michele Triestino}
\date{\today}

\begin{document}
\begin{abstract}

We show that the action of the  group $\Gamma=\langle a,b,c\mid a^2=b^3=c^7=abc\rangle$ on its space of left-orders has exactly two minimal components.

	\smallskip
	
	{\noindent\footnotesize \textbf{MSC\textup{2020}:} Primary 37C85, 57M60,37B05. Secondary 37E05.}

\end{abstract}

\maketitle

A \emph{left-order} on a group $G$ is a total order $\prec$ on $G$ which is preserved by the left-multiplication of $G$ on itself. A left-order is uniquely determined by its cone of positive elements $P_\prec=\{g\in G:g\succ \id\}$, which forms a semi-group, and satisfies $G=P_\prec\sqcup P^{-1}_\prec\sqcup\{\id\}$. Conversely, any semi-group satisfying such properties is the positive cone for some left-order. Because of this, one can see the set $\LO(G)$ of left-orders on $G$ as the set of positive cones on $G$, so that it is naturally a closed subset of $2^G$. 
A fundamental system of neighborhoods at a given left-order $\prec$ is given by the subsets
\[V_{\prec,F}=\{\prec'\in \LO(G):P_\prec\cap F=P_{\prec'}\cap F\},\]
where $F$ runs over finite subsets of $G$.

The conjugacy action of $G$ on itself induces an action by homeomorphisms on $\{0,1\}^G$ which preserves the subset of positive cones, and thus $G$ has a natural action by homeomorphisms on $\LO(G)$. Given an element $g\in G$ and a left-order $\prec\in \LO(G)$ we denote by $\prec_g$ the left-order obtained using this action. By definition, this is determined by the condition $P_{\prec_g}=gP_\prec g^{-1}$.

A general approach to understand the qualitative dynamical behavior of this action has been developed by Clay \cite{Clay}. Notably he proved (see also Rivas \cite{Rivas}) that in the case of nonabelian finite rank free groups, the action is topologically transitive, namely there exists a dense orbit.
When $G$ is a non-trivial group such that $\LO(G)$ is finite (such groups have been classified by Tararin, and they are a particular class of polycyclic groups), the action of $G$ on $\LO(G)$ has exactly two orbits. Using this, Clay gave an example of countable, not finitely generated, group for which the action is minimal, namely every orbit is dense. The problem of finding a non-trivial finitely generated group with this property is well-known to experts, see for instance Navas's survey \cite[Question 6]{NavasICM}.
One of the main difficulties is that there are often obstructions to make the orbit of a positive cone $P$ accumulate on $P^{-1}$ (see for instance \cite[Proposition 6]{Clay} for an obstruction related to convex generators for Conradian orderings).

The purpose of this note is to describe the action on the space of left-orders of the group 
\[
\Gamma:=\langle a,b,c\mid a^2=b^3=c^7=abc\rangle.
\]
By a classical result of Milnor \cite{Milnor}, this is isomorphic to the fundamental group of the Brieskorn homology 3-sphere $\Sigma(2,3,7)$, and it is also obtained by considering the $\Z$-central extension in $\widetilde{\PSL}(2,\R)$ of the cocompact Fuchsian $(2,3,7)$-triangle group $\Delta:=\Delta(2,3,7)$ in $\PSL(2,\R)$, as one sees from the standard presentation $\Delta\cong \langle a,b,c\mid a^2,b^3,c^7,abc\rangle$. (See \cite[Corollary 2.5]{Milnor}. Here we adopt a more common notation for these groups, as in Calegari \cite{Forcing}). The group $\Gamma$ is a quite remarkable left-orderable group, for instance Thurston \cite{Thurston}  noticed that it is an example of finitely generated, perfect left-orderable group.

We prove that the obstruction mentioned above is the only one for this group.
\begin{thm}\label{t.main}
	For any $\prec,\prec'\in \LO(\Gamma)$ such that $abc\in P_\prec\cap P_{\prec'}$, we have $\overline{\orb_\Gamma(\prec)}=\overline{\orb_\Gamma(\prec')}$. In particular, the action of $\Gamma$ on $\LO(\Gamma)$ has exactly two minimal components.
\end{thm}

Let us comment that $\Gamma$ gives the first known example of finitely generated group with this property, and admitting uncountably many left-orders. The proof of this result resembles somehow the proof of \cite[Proposition 2.12]{RivasTessera}, and it relies on the tight relation between left-orders and orientation-preserving actions on the real line. Indeed, given any left-order $\prec$ on a countable group $G$, one can construct a faithful action of $G$ on the line, by the so-called \emph{dynamical realization}. More precisely, there exists an action $\varphi_\prec:G\to \homeo_+(\R)$ and a base point $p\in \R$ such that $\varphi_\prec(g)(p)>p$ if and only if $g\in P_\prec$. 

\begin{rem}\label{r.dyn_real}
	Note that for any $g\in G$, the action $\varphi_\prec$ with base point $\varphi_\prec(g)(p)$ provides a dynamical realization of the left-order $\prec_g$.  This follows easily from the fact that the action of $G$ on the line is by orientation preserving homeomorphisms, so $\varphi_\prec(h)(p) > p \Leftrightarrow \varphi_\prec(gh)(p) > \varphi_\prec(g)(p) \Leftrightarrow \varphi_\prec(ghg^{-1})\varphi_\prec(g)(p) > \varphi_\prec(g)(p)$.
\end{rem}

Being naturally a subgroup of $\widetilde{\PSL}(2,\R)$, the group $\Gamma$ admits a faithful action on the real line.
More precisely,  the action of $\Gamma$ is the $\Z$-central lift of the action of $\Delta\in \PSL(2,\R)$ by M\"obius transformations on $\R\mathrm{P}^1$ (or, equivalently, the action on the boundary of the hyperbolic disk $\partial \mathbb H^2$ induced by isometries of $\mathbb H^2$),  where the center is generated by the element $abc$.

\begin{rem}\label{r.hyperbolic}
	Thus, this action has the property that every point stabilizer is either trivial or cyclic (because $\Delta$ is a discrete subgroup of $\PSL(2,\R)$), and in the latter case one can find a generator of the stabilizer which contracts any sufficiently small neighborhood of the point (because $\Delta$ is cocompact in $\PSL(2,\R)$).
\end{rem}

We will refer to the action of $\Gamma$ coming from the embedding on $\widetilde{\PSL}(2,\R)$, as to the \emph{standard action} $\rho:\Gamma\to \homeo_+(\R)$. The remarkable (and well-known) fact that we will use is that the standard action is essentially the unique action of $\Gamma$  on the line.

\begin{thm}\label{t.unique}
	Let $\varphi:\Gamma\to \homeo_+(\R)$ be an action without global fixed points. Then there exists a continuos monotone map $h:\R\to \R$ such that $h\circ \varphi(g)=\rho(g)\circ h$ for every $g\in \Gamma$.
\end{thm}

\begin{proof}
	As explained in Calegari \cite[Remark 4.3.4]{Forcing}, by a combination of the Milnor--Wood inequality and a theorem of Matsumoto
	the group $\Delta$ admits a unique non-trivial action on the circle up to semi-conjugacy.  We claim that for any action $\varphi$ of $\Gamma$ on $\R$ as in Theorem \ref{t.unique}, the central element $abc$ acts without fixed points.  Thus, there is an induced action of $\Delta=\Gamma/\langle abc\rangle$ on the circle $\R/\langle \varphi(abc)\rangle$, which is nontrivial (else $abc$ would act trivially on the line), so unique up to semi-conjugacy.  Since the standard action is minimal, such a semi-conjugacy can be realized by a continuous monotone map (projected to the circle) as in the statement of the theorem, showing that $\Gamma$ has a unique action up to semi-conjugacy as desired.  
	
	To prove the claim, assume that $p\in \R$ is fixed by $\varphi(abc)$. Since we are assuming that $\varphi$ has no global fixed points, there must be a generator $s\in \{a,b,c\}$ such that $\varphi(s)(p)\neq p$. Assume, say, that $s=a$ and $\varphi(a)(p)>p$ (the other cases can be treated similarly). Then we have $\varphi(a^2)(p)>\varphi(a)(p)>p$, but $\varphi(a^2)(p)=\varphi(abc)(p)=p$, contradicting the assumption.
\end{proof}

The conclusion in the statement is often rephrased by saying that $\varphi$ and $\rho$ are \emph{semi-conjugate} actions. We will say that they are \emph{positively} semi-conjugate if the map $h$ is increasing. When $h$ is not injective, we also say that $\varphi$ is a \emph{blow up} of $\rho$.

\begin{rem}\label{r.blowup}
	In general, if $\psi:G\to \homeo_+(\R)$ is a blow up of a dynamical realization $\varphi_\prec$ with base point $p$, with positive semi-conjugacy $h:\R\to \R$, then for any $q \in h^{-1}(p)$, the linear order of points in the orbit $\psi(G)(q)$ agrees with that of $\varphi_\prec(G)(p)$.
\end{rem}

\begin{proof}[Proof of Theorem \ref{t.main}]
	
	Fix two left-orders $\prec$ and $\prec'$ on $\Gamma$ and a finite subset $F\subset \Gamma$. We write $P$ and $P'$ for the corresponding positive cones. Assume that $abc\in P\cap P'$. We want to find an element $g\in \Gamma$ such that $\prec_g'$ belongs to the neighborhood $V_{\prec,F}$.
	Let $\varphi:\Gamma\to \homeo_+(\R)$ and $\varphi':\Gamma\to \homeo_+(\R)$ be dynamical realizations of $\prec$ and $\prec'$ respectively, with corresponding base points $p$ and $p'$.
	As $abc\in P\cap P'$, by Theorem \ref{t.unique}, the dynamical realizations $\varphi$ and $\varphi'$ are positively semi-conjugate. As any two positively semi-conjugate actions admit a common blow up (see for instance \cite[Theorem 2.2]{KKMj}), by Remark \ref{r.blowup} we can actually assume $\varphi=\varphi'$, and we will simply write $g(x)$ instead of $\varphi(g)(x)$ for any $g\in \Gamma$, $x\in \R$. We denote by $\Lambda\subset \R$ the minimal invariant closed subset for the action $\varphi$.
	When $p$ or $p'$ is not in $\Lambda$, we let $J$ (respectively, $J'$) be the closure of the connected component of $\R\setminus \Lambda$ containing the base point $p$ (respectively, $p'$), and we write $K=\stab(J)$ (respectively, $K'=\stab(J')$). In case $p$ or $p'$ is in $\Lambda$, we simply put $J=\{p\}$ (or $J'=\{p'\}$), in which case $K$ (respectively, $K'$) is trivial. Note that as for the standard action $\rho$ point stabilizers are either trivial or cyclic (Remark \ref{r.hyperbolic}), the subgroup $K$ (and so $K'$) is either trivial or cyclic. 
	
	As an easy first case, assume to start that $K$ is trivial. By continuity of the action, we can find a neighborhood $V$ of $J$, such that $f(x)>x$ if and only if $f(p)>p$ for every $f\in F$ and $x\in V$. 
	Let $g\in \Gamma$ be any element such that $g(p')\in V$, in this case, by Remark \ref{r.dyn_real}, we have $\prec_g'\in V_{\prec,F}$, as wanted.
	
	When $K$ is non-trivial, we need to take some care to choose such an element $g$ to send $p'$ to the correct side of $J$.  
	By Remark \ref{r.hyperbolic}, we can take a generator $k$ of $K$ such that for every sufficiently small neighborhood $U$ of $J$, one has $k(U)\subset U$.  That is to say, if $u$ is close to $J$ but strictly to the left of $J$, we have $k(u) > u$, and if $u$ is close to $J$ but strictly to the right, then $k(u) < u$.  
	Suppose first that $k\in P$.  Then choose $g \in \Gamma$ such that $g(p')$ lies to the left of $J$ in $U - J$, chosen sufficiently close so that $f(x)>x$ if and only if $f(p)>p$ for any $f\in F$ (note that our choice of the left side ensures this holds even if $K \cap F \neq \emptyset$), 
	and such that $gK'g^{-1}\cap F=\emptyset$.  The last condition may easily be satisfied because each side of $U-J$ intersects the minimal set for the action.  Then, by Remark \ref{r.dyn_real} again, $\prec_g'\in V_{\prec,F}$, as desired.   If instead we have $k \in P^{-1}$, one makes the same argument choosing $g(p')$ to be to the right.  
\end{proof}

As a final comment, let us mention that other finitely generated groups are known to admit a unique nontrivial action on the line, up to semi-conjugacy. One interesting example is provided by the central extension $\widetilde T$ of Thompson's group $T$ of dyadic piecewise linear circle homeomorphisms (see the work of Matte Bon and the second author \cite{MatteBonTriestino}). Understanding this example was our original motivation, as we were wondering whether one could find examples of groups whose action on the space of orders is minimal by considering the groups of Hyde and Lodha \cite{HydeLodha} (or their more conceptual generalization presented by Le Boudec and Matte Bon in \cite{LBMB}), which are ``quasi-periodic'' versions of $\widetilde T$.
However, point stabilizers for the standard action of $\widetilde T$ are very large, as they are always isomorphic to a subgroup of Thompson's group $F$ containing $[F,F]$, and the classification of actions of such groups (even up to semi-conjugacy, see the work of Brum, Matte Bon, Rivas and the second author \cite{BMRT}) is quite complex. Because of this, the classification of actions of $\widetilde T$ \emph{up to conjugacy} is also complex, and there is little hope to understand the minimal components of the action of $\widetilde T$ on $\LO(\widetilde T)$ (and even worse in the case of the groups of Hyde and Lodha and generalizations).

{\small \subsection*{Acknowledgments}
M.T.\ thanks the warm hospitality of the Department of Mathematics at Cornell University. 
K.M. was partially supported by NSF grant DMS-1844516 and a Sloan Fellowship.
M.T. is partially supported by the project ANR
Gromeov (ANR-19-CE40-0007), the project ANER
Agroupes (AAP 2019 Région Bourgogne–Franche–Comté), and his host department IMB
receives support from the EIPHI Graduate School (ANR-17-EURE-0002).}

\bibliographystyle{plain}

\bibliography{biblio.bib}

\begin{bibdiv}
\begin{biblist}

\bib{BMRT}{misc}{
      author={Brum, Joaquín},
      author={Bon, Nicolás~Matte},
      author={Rivas, Cristóbal},
      author={Triestino, Michele},
       title={Locally moving groups acting on the line and $\mathbb{R}$-focal
  actions},
        date={2021},
        note={arXiv:2104.14678},
}

\bib{Forcing}{article}{
      author={Calegari, Danny},
       title={Dynamical forcing of circular groups},
        date={2006},
        ISSN={0002-9947},
     journal={Trans. Amer. Math. Soc.},
      volume={358},
      number={8},
       pages={3473\ndash 3491},
         url={https://doi.org/10.1090/S0002-9947-05-03754-2},
      review={\MR{2218985}},
}

\bib{Clay}{article}{
      author={Clay, Adam},
       title={Free lattice-ordered groups and the space of left orderings},
        date={2012},
        ISSN={0026-9255},
     journal={Monatsh. Math.},
      volume={167},
      number={3-4},
       pages={417\ndash 430},
         url={https://doi.org/10.1007/s00605-011-0305-5},
      review={\MR{2961291}},
}

\bib{HydeLodha}{article}{
      author={Hyde, James},
      author={Lodha, Yash},
       title={Finitely generated infinite simple groups of homeomorphisms of
  the real line},
        date={2019},
        ISSN={0020-9910},
     journal={Invent. Math.},
      volume={218},
      number={1},
       pages={83\ndash 112},
         url={https://doi.org/10.1007/s00222-019-00880-7},
      review={\MR{3994586}},
}

\bib{KKMj}{book}{
      author={Kim, Sang-hyun},
      author={Koberda, Thomas},
      author={Mj, Mahan},
       title={Flexibility of group actions on the circle},
      series={Lecture Notes in Mathematics},
   publisher={Springer, Cham},
        date={2019},
      volume={2231},
        ISBN={978-3-030-02854-1; 978-3-030-02855-8},
         url={https://doi.org/10.1007/978-3-030-02855-8},
      review={\MR{3887602}},
}

\bib{LBMB}{misc}{
      author={Le~Boudec, Adrien},
      author={Matte~Bon, Nicolás},
       title={Confined subgroups and high transitivity},
        date={2020},
        note={arXiv:2012.03997},
}

\bib{MatteBonTriestino}{article}{
      author={Matte~Bon, Nicol\'as},
      author={Triestino, Michele},
       title={Groups of piecewise linear homeomorphisms of flows},
        date={2020},
     journal={Compositio Math.},
      volume={156},
      number={8},
       pages={1595\ndash 1622},
}

\bib{Milnor}{incollection}{
      author={Milnor, John},
       title={On the {$3$}-dimensional {B}rieskorn manifolds {$M(p,q,r)$}},
        date={1975},
   booktitle={Knots, groups, and 3-manifolds ({P}apers dedicated to the memory
  of {R}. {H}. {F}ox)},
      series={Ann. of Math. Studies, No. 84},
   publisher={Princeton Univ. Press, Princeton, N. J.},
       pages={175\ndash 225},
      review={\MR{0418127}},
}

\bib{NavasICM}{inproceedings}{
      author={Navas, Andr\'{e}s},
       title={Group actions on 1-manifolds: a list of very concrete open
  questions},
        date={2018},
   booktitle={Proceedings of the {I}nternational {C}ongress of
  {M}athematicians---{R}io de {J}aneiro 2018. {V}ol. {III}. {I}nvited
  lectures},
   publisher={World Sci. Publ., Hackensack, NJ},
       pages={2035\ndash 2062},
      review={\MR{3966841}},
}

\bib{Rivas}{article}{
      author={Rivas, Crist\'{o}bal},
       title={Left-orderings on free products of groups},
        date={2012},
        ISSN={0021-8693},
     journal={J. Algebra},
      volume={350},
       pages={318\ndash 329},
         url={https://doi.org/10.1016/j.jalgebra.2011.10.036},
      review={\MR{2859890}},
}

\bib{RivasTessera}{article}{
      author={Rivas, Crist\'{o}bal},
      author={Tessera, Romain},
       title={On the space of left-orderings of virtually solvable groups},
        date={2016},
        ISSN={1661-7207},
     journal={Groups Geom. Dyn.},
      volume={10},
      number={1},
       pages={65\ndash 90},
         url={https://doi.org/10.4171/GGD/343},
      review={\MR{3460331}},
}

\bib{Thurston}{article}{
      author={Thurston, William~P.},
       title={A generalization of the {R}eeb stability theorem},
        date={1974},
        ISSN={0040-9383},
     journal={Topology},
      volume={13},
       pages={347\ndash 352},
         url={https://doi.org/10.1016/0040-9383(74)90025-1},
      review={\MR{356087}},
}

\end{biblist}
\end{bibdiv}

\medskip

\noindent\textit{Kathryn Mann\\
	Department of Mathematics,
	Cornell University,
	Ithaca, NY 14853, USA\\}
\href{mailto:k.mann@cornell.edu}{k.mann@cornell.edu}

\medskip

\noindent\textit{Michele Triestino\\
	Institut de Math\'ematiques de Bourgogne (UMR CNRS 5584),
	Universit\'e de Bourgogne,
	9 av.~Alain Savary, 21000 Dijon, France\\}
\href{mailto:michele.triestino@u-bourgogne.fr}{michele.triestino@u-bourgogne.fr}
\end{document}